\UseRawInputEncoding
\documentclass{amsart}[12pt]
\usepackage{graphicx,graphics, amsmath, amsthm, amscd, amsfonts}

\usepackage{latexsym}
\makeatletter \oddsidemargin.9375in \evensidemargin \oddsidemargin
\newtheorem{theorem}{Theorem}[section]
\newtheorem{lemma}[theorem]{Lemma}

\theoremstyle{definition}

\newtheorem{remark}[theorem]{Remark}
\numberwithin{equation}{section}

\begin{document}
\Large
\title[Maps completely preserving the quadratic operators]{Maps completely preserving the quadratic operators}

\author[ Roja Hosseinzadeh]{Roja Hosseinzadeh}

\address{{ Department of Mathematics, Faculty of Mathematical Sciences,
 University of Mazandaran, P. O. Box 47416-1468, Babolsar, Iran.}}

\email{ro.hosseinzadeh@umz.ac.ir}

\subjclass[2000]{46J10, 47B48}

\keywords{Completely preserving problem, Quadratic operator, Operator algebra.}

\begin{abstract}\large
Let $\mathcal{A}$ and $\mathcal{B}$ be standard operator algebras on Banach spaces $\mathcal{X}$ and $\mathcal{Y}$, respectively. In this paper, we show that every bijection completely preserving quadratic operators from $\mathcal{A}$ onto $\mathcal{B}$ is either an isomorphism or (in the complex case) a conjugate isomorphism.
\end{abstract} \maketitle

\section{Introduction}
\noindent

The study of maps on operator algebras
preserving certain properties or subsets is a subject that has attracted the attention of many mathematicians.
To get acquainted with this topic and as an introduction to it, you can see reference \cite{8} and references inside it.
Recently, some of these problems have been related to the completely preserving of certain properties or subsets of operators. For example see $[2-7]$.

Let $\mathcal{B(X)}$ be the algebra of all
bounded linear operators on a Banach space $\mathcal{X}$. Recall that an operator $T \in \mathcal{B(X)}$ is quadratic if there exist two complex
numbers $a, b \in \mathbb{C}$, such that $(T-aI)(T -bI) = 0$.

Let $ \phi : \mathcal{M} \rightarrow \mathcal{N}$ be a map, where $ \mathcal{M}$ and $ \mathcal{N}$ are linear spaces. Define for each $n \in \mathbb{N} $, a map $ \phi _n : \mathcal{M} \otimes M_n \rightarrow \mathcal{N} \otimes M_n$ by
$$ \phi _n ((a_{ij})_{n \times n})=( \phi (a_{ij}))_{n \times n}.$$
Then $ \phi $ is said to be $n$-quadratic preserving if $ \phi _n $ preserves quadratic operators. $ \phi $ is said to be completely quadratic preserving if $ \phi _n $ preserves quadratic operators for each $n \in \mathbb{N} $.

In \cite{2}, authors characterized completely rank-nonincreasing linear maps and then later extended on \cite{3}.
Completely invertibility preserving maps were characterized in \cite{5}. Subsequently, in \cite{6} completely idempotent and completely
square-zero preserving maps and in \cite{7} completely commutativity and
completely Jordan zero product preserving maps were discussed.

Recall that a standard operator
algebra on $\mathcal{X}$ is a norm closed subalgebra of $\mathcal{B(X)}$ which
contains the identity and all finite rank operators. Our main
results is as follows.

\begin{theorem} Let $\mathcal{X}$, $\mathcal{Y}$ be infinite dimensional Banach spaces and $\mathcal{A}$ and $\mathcal{B}$ be standard operator algebras on $\mathcal{X}$ and $\mathcal{Y}$, respectively. Let $\phi: \mathcal{A} \rightarrow \mathcal{B}$ be a bijective map. Then the following statements are equivalent:

$(1)$ $\phi$ is completely preserving the quadratic operators in both directions.

$(2)$ $\phi$ is 2-quadratic preserving operators in both directions.

$(3)$ There exist a bounded invertible linear or (in the complex case) conjugate-
linear operator $A : \mathcal{X} \rightarrow \mathcal{Y}$ and a scalar $ \lambda $ such that
$$ \phi(T ) = \lambda ATA^{-1},$$
 for all $T \in \mathcal{A}$.
 \end{theorem}
\begin{theorem}  Let $\phi: \mathcal{M}_n( \mathbb{F} ) \rightarrow \mathcal{M}_n( \mathbb{F} )$ $(n \geq 3)$ be a bijective map, where $\mathbb{F}$ is the real or complex field. Then the following statements are equivalent:

$(1)$ $\phi$ is completely preserving the quadratic operators in both directions.

$(2)$ $\phi$ is 2-quadratic preserving operators in both directions.

$(3)$ There exist an invertible matrix $A \in \mathcal{M}_n$, a scalar $ \lambda $ and an automorphism $ \tau :\mathbb{F} \rightarrow \mathbb{F}$ such that
$$ \phi(T ) = \lambda AT_{\tau}A^{-1},$$
 for all $T \in \mathcal{M}_n( \mathbb{F} )$. Here $T_{\tau}=( \tau (t_{ij}))$ for $T=(t_{ij})$.
\end{theorem}

\section{Proofs}

Denote by $\mathcal{X}^ *$ the dual space of $\mathcal{X}$. For every nonzero $x\in \mathcal{X}$ and
$f\in \mathcal{X}^ *$, the symbol $x\otimes f$ stands for the rank one
linear operator on $\mathcal{X}$ defined by $(x\otimes {f})y=f(y)x$ for any
$y\in \mathcal{X}$. Note that every rank one operator in $\mathcal{B(X)}$ can be
written in this way. The rank one operator $x\otimes f$ is
idempotent (resp. nilpotent) if and only if $f(x)=1$ (resp. $f(x)=0$).
Let $P$ and $Q$ be two idempotent operators. We say that $P$ and $Q$ are orthogonal if and only if
$PQ=QP=0$. Let $\mathcal{A}$ be a standard operator
algebra on $\mathcal{X}$. Denote by $\mathcal{P}_{ \mathcal{A}}$ the set of all idempotent in $\mathcal{A}$.
 In order to prove the main results, we need to some auxiliary lemmas.

\begin{lemma}
$\phi(0)=0$ and $ \phi ( \mathbb{C}I)= \mathbb{C}I$, specially $ \phi (I)= \lambda I$ for some constant $\lambda$.
\label{2.1}
\end{lemma}
\begin{proof}  Let $ c \in \mathbb{C}$ be a nonzero number and $T$ be an arbitrary operator. Set
 $$ A=\begin{pmatrix}
     cI & T \\
     0 & 0 \\
   \end{pmatrix}.$$
   It is clear that $A(A- \begin{pmatrix}
     cI & 0 \\
     0 & cI \\
   \end{pmatrix})=0$ and then there exist constants $a,b$ such that $ (\phi _2(A)- \begin{pmatrix}
     aI & 0 \\
     0 & aI \\
   \end{pmatrix})( \phi _2(A)- \begin{pmatrix}
     bI & 0 \\
     0 & bI \\
   \end{pmatrix})=0$. This implies
$$ \begin{pmatrix}
     \phi(cI)-aI & \phi(T) \\
     \phi(0) & \phi(0)-aI \\
   \end{pmatrix} \begin{pmatrix}
     \phi(cI)-bI & \phi(T) \\
     \phi(0) & \phi(0)-bI \\
   \end{pmatrix}=\begin{pmatrix}
     0 & 0 \\
     0 & 0 \\
   \end{pmatrix}$$
and hence
\begin{eqnarray}
 ( \phi(cI)-aI)( \phi (cI)-bI)+ \phi(T) \phi (0)=0,
\end{eqnarray}
\begin{eqnarray}
 ( \phi(cI)-aI)\phi(T)+\phi(T)( \phi (cI)-bI)=0.
\end{eqnarray}
Surjectivity of $\phi$ yields that there exist $T_0, T_1 \in \mathcal{A}$ such that $\phi(T_0)=0$ and $\phi(T_1)=I$. Taking $T=T_0$ in Equation $(2.1)$ yields that $( \phi(cI)-aI)( \phi (cI)-bI)=0$ and then $\phi(T) \phi (0)=0$. Taking $T=T_1$ in last equation yields $\phi (0)=0$. Again taking  $T=T_1$ in Equation $(2.2)$ implies $\phi(cI)-aI+\phi (cI)-bI=0$ and therefore $\phi(cI)=\frac{a+b}{2}I$. Thus $ \phi ( \mathbb{C}I) \subseteq \mathbb{C}I$.
Since $\phi$ is bijective and $\phi ^{-1}$ has the all of properties of $\phi$, then we can conclude that $ \phi ( \mathbb{C}I)= \mathbb{C}I$
\end{proof}
\begin{remark}
It is clear that $ \lambda ^{-1} \phi$ satisfies in assumptions on $\phi$ and also by previous lemma is unital. Thus without loss of generality, in the following lemmas, we assume that $ \phi (I)=I$.
\label{2.2}
\end{remark}
\begin{lemma}
$\phi$ preserves the idempotent operators in both directions.
\label{2.3}
\end{lemma}
\begin{proof}
Let $P$ be an idempotent. If $P=0$ or $I$, then by previous lemma and remark, the assertion is true. So let $P$ be a non-trivial idempotent. It is clear that $ R=\begin{pmatrix}
     P & 0 \\
     0 & 0 \\
   \end{pmatrix}$
    is idempotent. So by assumption and Lemma \ref{2.1}, there exist constants $a_1,b_1$ such that
   $$ (\phi _2(R)- \begin{pmatrix}
     a_1I & 0 \\
     0 & a_1I \\
   \end{pmatrix})( \phi _2(R)- \begin{pmatrix}
     b_1I & 0 \\
     0 & b_1I \\
   \end{pmatrix})=0$$
   $$ \Rightarrow \begin{pmatrix}
     \phi(P)-a_1I & 0 \\
     0 & -a_1I \\
   \end{pmatrix} \begin{pmatrix}
     \phi(P)-b_1I & 0 \\
     0 & -b_1I \\
   \end{pmatrix}=\begin{pmatrix}
     0 & 0 \\
     0 & 0 \\
   \end{pmatrix}$$
   which implies
   \begin{eqnarray*}
 ( \phi(P)-a_1I)( \phi (P)-b_1I)=0,
\end{eqnarray*}
\begin{eqnarray*}
 a_1b_1I=0.
\end{eqnarray*}
These two relations imply
\begin{eqnarray}
 \phi(P)^2=(a_1+b_1)\phi(P).
\end{eqnarray}
It is clear that $ S=\begin{pmatrix}
     P & 0 \\
     0 & I \\
   \end{pmatrix}$ is idempotent. So by assumption and Lemma \ref{2.1}, there exist constants $a_2,b_2$ such that
   $$ (\phi _2(S)- \begin{pmatrix}
     a_2I & 0 \\
     0 & a_2I \\
   \end{pmatrix})( \phi _2(S)- \begin{pmatrix}
     b_2I & 0 \\
     0 & b_2I \\
   \end{pmatrix})=0$$
   $$ \Rightarrow \begin{pmatrix}
     \phi(P)-a_2I & 0 \\
     0 & (1-a_2)I \\
   \end{pmatrix} \begin{pmatrix}
     \phi(P)-b_2I & 0 \\
     0 & (1-b_2)I \\
   \end{pmatrix}=\begin{pmatrix}
     0 & 0 \\
     0 & 0 \\
   \end{pmatrix}$$
   which implies
   \begin{eqnarray}
 ( \phi(P)-a_1I)( \phi (P)-b_1I)=0,
\end{eqnarray}
 \begin{eqnarray}
  (1-a_2)(1-b_2)I=0.
\end{eqnarray}
$(2.4)$ together with $(2.3)$ implies
\begin{eqnarray}
 (a_2+b_2-a_1-b_1)\phi(P)=a_2b_2I.
\end{eqnarray}
We assert that $a_2+b_2-a_1-b_1=0$. Otherwise, $\phi(P)$ is a nonzero multiple of identity which Lemma \ref{2.1} yields that $P$ is a nonzero multiple of identity. This is a contradiction, because $P$ is a non-trivial idempotent.

Thus $a_2+b_2-a_1-b_1=0$ and then by $(2.6)$, $a_2b_2=0$. So from $(2.5)$, $a_2+b_2=1$ which implies $a_1+b_1=1$. This and $(2.3)$ follow the idempotency of $\phi(P)$. Since $\phi$ is bijective and $\phi ^{-1}$ has the all of properties of $\phi$, then we can conclude that $\phi$ preserves the idempotent operators in both directions.
\end{proof}
\begin{lemma}
$\phi$ preserves the orthogonality of idempotent operators in both directions.
\label{2.4}
\end{lemma}
\begin{proof}
Let $P$ and $Q$ be non-trivial orthogonal idempotents. By Lemmas \ref{2.1} and \ref{2.3}, $\phi (P)$ and $\phi(Q)$ are non-trivial idempotents. So it is enough to show that $\phi(P) \phi(Q)=\phi(Q)\phi(P)=0$.
It is cleat that
$ R^2- \begin{pmatrix}
     I & 0 \\
     0 & I \\
   \end{pmatrix}=0$, where $R=\begin{pmatrix}
     I & P \\
     Q & -I \\
   \end{pmatrix}$. By Lemma \ref{2.1}, there exist a scalar $\alpha$ such that $\phi(-I)= \alpha $.
    So by assumption and Lemma \ref{2.1}, there exist constants $a,b$ such that
   $$ (\phi _2(R)- \begin{pmatrix}
     aI & 0 \\
     0 & aI \\
   \end{pmatrix})( \phi _2(R)- \begin{pmatrix}
     bI & 0 \\
     0 & bI \\
   \end{pmatrix})=0$$
$$ \Rightarrow \begin{pmatrix}
     (1-a)I & \phi(P) \\
     \phi(Q) & (\alpha-a)I \\
   \end{pmatrix} \begin{pmatrix}
     (1-b)I & \phi(P) \\
     \phi(Q) & (\alpha-b)I \\
   \end{pmatrix}=\begin{pmatrix}
     0 & 0 \\
     0 & 0 \\
   \end{pmatrix}$$
   which implies
   \begin{eqnarray}
 ( 1-a)(1-b)I+ \phi(P) \phi(Q)=0,
\end{eqnarray}
\begin{eqnarray}
 ( \alpha-a)(\alpha-b)I+ \phi(Q) \phi(P)=0.
\end{eqnarray}
Multiplying $(2.7)$ from left by $\phi(P)$, we see that
\begin{eqnarray*}
 ( 1-a)(1-b) \phi(P)+ \phi(P) \phi(Q)=0.
\end{eqnarray*}
This equation together with $(2.7)$ entails
\begin{eqnarray*}
 (1-a)(1-b) \phi(P)=( 1-a)(1-b)I.
\end{eqnarray*}
If $( 1-a)(1-b) \neq 0$, then $\phi(P)=I$ which is a contradiction.
Hence $(1-a)(1-b)=0$ and then by $(2.7)$, $\phi(P) \phi(Q)=0$. Using the same method on equation $(2.8)$, follows $\phi(Q) \phi(P)=0$. Thus $\phi$ preserves the orthogonality of idempotents.
Since $\phi$ is bijective and $\phi ^{-1}$ has the all of properties of $\phi$, then we can conclude that $\phi$ preserves the orthogonality of idempotents in both directions.
\end{proof}
\begin{lemma}
$ \phi (P)= APA^{-1}$, for every $P \in \mathcal{P}_{ \mathcal{A}}$, where $A:X\longrightarrow Y$ is a bijective bounded linear operator.
\label{2.5}
\end{lemma}
\begin{proof} Lemmas \ref{2.4} implies that $\phi$ is a bijection preserving the orthogonality of idempotents in both directions. It follows from lemma 3.1 in $[10]$ that there exists a bijective bounded linear or (in the complex case) conjugate linear operator $A:X\longrightarrow Y$ such that
$$\phi(P)=APA^{-1} \hspace{.4cm} (P \in \mathcal{P}_{ \mathcal{A}})$$
or a bijective bounded linear or (in the complex case) conjugate linear operator $A:X ^{ \prime}\longrightarrow Y$ such that
$$\phi(P)=AP^*A^{-1} \hspace{.4cm} (P \in \mathcal{P}_{ \mathcal{A}}).$$
We show that the second case can not occur. Assume on the contrary that $\phi(P)=AP^*A^{-1}$ for all $P \in \mathcal{P}_{ \mathcal{A}}$. Let $x,y \in X$ be linearly independent vectors. So there exist $f,g \in X ^{ \prime}$ then we have $f(x)=1$,$f(y)=-1$ and $g(x)=g(y)=1$. Hence we have
$ R^2- \begin{pmatrix}
     I & 0 \\
     0 & I \\
   \end{pmatrix}=0$, where $R=\begin{pmatrix}
     I-x \otimes f & x \otimes g \\
      -y \otimes f & I-y \otimes g \\
    \end{pmatrix} $.
Thus there exist constants $a,b$ such that
   $$ (\phi _2(R)- \begin{pmatrix}
     aI & 0 \\
     0 & aI \\
   \end{pmatrix})( \phi _2(R)- \begin{pmatrix}
     bI & 0 \\
     0 & bI \\
   \end{pmatrix})=0.$$
It is clear that $I-x \otimes f$, $ x \otimes g$, $-y \otimes f$  and $I-y \otimes g $ are idempotents and so
\begin{eqnarray*}ý
 \begin{array}{lcl}ý
  \begin{pmatrix}
A(I-f \otimes x)A^{-1}-aI & A(g \otimes x)A^{-1} \\
      -A(f \otimes y)A^{-1} & A(I-g \otimes y)A^{-1}-aI \\
   \end{pmatrix}\\
   \times \begin{pmatrix}
     A(I-f \otimes x)A^{-1}-bI & A(g \otimes x)A^{-1} \\
      -A(f \otimes y)A^{-1} & A(I-g \otimes y)A^{-1}-bI \\
   \end{pmatrix}=\begin{pmatrix}
     0 & 0 \\
     0 & 0 \\
   \end{pmatrix}
   \end{array}
 ý\end{eqnarray*}
 \begin{eqnarray*}ý
 \begin{array}{lcl}
   \Rightarrow  \begin{pmatrix}
     A & 0 \\
     0 & A \\
   \end{pmatrix} \begin{pmatrix}
     (1-a)I-f \otimes x & g \otimes x \\
      -f \otimes y & (1-a)I-g \otimes y \\
   \end{pmatrix} \begin{pmatrix}
     A^{-1} & 0 \\
     0 & A^{-1} \\
   \end{pmatrix}\\
   \times \begin{pmatrix}
     A & 0 \\
     0 & A \\
   \end{pmatrix} \begin{pmatrix}
     (1-b)I-f \otimes x & g \otimes x \\
      -f \otimes y & (1-b)I-g \otimes y \\
   \end{pmatrix} \begin{pmatrix}
     A^{-1} & 0 \\
     0 & A^{-1} \\
   \end{pmatrix}\\
   =\begin{pmatrix}
     0 & 0 \\
     0 & 0 \\
   \end{pmatrix}
 ý\end{array}
 ý\end{eqnarray*}
 \begin{eqnarray*}ý
 \begin{array}{lcl}
   \Rightarrow  \begin{pmatrix}
     (1-a)I-f \otimes x & g \otimes x \\
      -f \otimes y & (1-a)I-g \otimes y \\
   \end{pmatrix}\\
     \times \begin{pmatrix}
     (1-b)I-f \otimes x & g \otimes x \\
      -f \otimes y & (1-b)I-g \otimes y \\
   \end{pmatrix}
   =\begin{pmatrix}
     0 & 0 \\
     0 & 0 \\
   \end{pmatrix}.
 ý\end{array}
 ý\end{eqnarray*}
One of the equations that can be obtained from this relation is the following equation:
$$((1-a)I-f \otimes x)((1-b)I-f \otimes x)-(g \otimes x)(f \otimes y)=0.$$
This implies
$$(a+b-1)f \otimes x+g \otimes y+(a-1)(b-1)I=0$$
and then
$$(a+b-1)x \otimes f+y \otimes g+(a-1)(b-1)I=0$$
$$ \Rightarrow (a+b-1)(x \otimes f)x+(y \otimes g)x+(a-1)(b-1)x=0$$
$$ \Rightarrow abx+y=0,$$
which is a contradiction, because $x$ and $y$ are linearly independent.
Therefore, the second case can not occur and the proof is completed.
\end{proof}
\begin{lemma}
For every operator $T$, there exists a complex number $ \lambda _T$ such that $ \phi (T)= \lambda _TT $.
\label{2.6}
\end{lemma}
\begin{proof} By Lemma \ref{2.5}, $ \phi (P)= APA^{-1}$, for every idempotent operator $P$, where $A:X\longrightarrow Y$ is a bijective bounded linear operator.  It is trivial that without loss of generality, we can suppose that $\phi(P)=P$ for all $P \in \mathcal{P}_{ \mathcal{A}}$. \par For every $T \in \mathcal{A}$ we have
$$
    \begin{pmatrix}
     T & I \\
     I-T^2 & -T \\
    \end{pmatrix} ^2 -  \begin{pmatrix}
     I & 0 \\
     0 & I \\
    \end{pmatrix} = 0.$$
    Thus there exist constants $a,b$ such that
   $$ (\phi _2( \begin{pmatrix}
   T & I \\
     I-T^2 & -T
     \end{pmatrix})- \begin{pmatrix}
     aI & 0 \\
     0 & aI \\
   \end{pmatrix})( \phi _2( \begin{pmatrix}
   T & I \\
     I-T^2 & -T
     \end{pmatrix})- \begin{pmatrix}
     bI & 0 \\
     0 & bI \\
   \end{pmatrix})=0.$$
    One of the equations that can be obtained from this relation is the following equation:
    \begin{eqnarray}
\phi (-T)+ \phi (T)=(a+b)I,
\end{eqnarray}
 which implies $a+b=0$, because otherwise by Lemma \ref{2.1}, there exist a nonzero scalar $t$ such that $\phi (tI)=(a+b)I$. By replacing $tI$ instead $T$ in equation $(2.9)$ , we get $\phi (-tI)=0$ which implies $t=0$, a contradiction. Therefore
 \begin{eqnarray}
\phi (-T)= -\phi (T).
\end{eqnarray}
On the other hand, for any $T,S \in \mathcal{A}$ we have
$$
 \begin{pmatrix}
     I-TS & -T \\
     STS-2S & ST-I \\
   \end{pmatrix} ^2 -  \begin{pmatrix}
     I & 0 \\
     0 & I \\
    \end{pmatrix} = 0.$$
   Hence there exist constants $c,d$ such that
   \begin{eqnarray*}ý
 \begin{array}{lcl}
   (\phi _2( \begin{pmatrix}
   I-TS & -T \\
     STS-2S & ST-I \\
     \end{pmatrix})- \begin{pmatrix}
     cI & 0 \\
     0 & cI \\
   \end{pmatrix})\\
     \times ( \phi _2( \begin{pmatrix}
   I-TS & -T \\
     STS-2S & ST-I \\
     \end{pmatrix})- \begin{pmatrix}
     dI & 0 \\
     0 & dI \\
   \end{pmatrix})=0.
 ý\end{array}
 ý\end{eqnarray*}
   One of the equations that can be obtained from this relation is the following equation:
    $$ [\phi (I-TS)-cI] \phi (-T)+ \phi (-T) [\phi (ST-I)-dI]=0.$$
   Let $T$ be an operator and $x$ be an arbitrary nonzero vector of $X$ such that $x \not \in \ker T$. So there exists a nonzero functional $f$ such that $f(Tx)=1$. If $S=x \otimes f$, then $I-TS$ is an idempotent and so $\phi (I-TS)=I-TS$ and also by $(2.10)$, $\phi (TS-I)=TS-I$.
   Setting $S=x \otimes f$ in previous equation and using $(2.10)$ yields
   \begin{eqnarray*}ý
 \begin{array}{lcl}
   ([(1-c)I-Tx \otimes f] \phi (T)+ \phi (T)[x \otimes fT-(1+d)I]=0\\
     \Rightarrow (c+d) \phi (T)+ Tx \otimes f \phi (T)- \phi (T)x \otimes fT=0,
 ý\end{array}
 ý\end{eqnarray*}
   which putting $T=I$ follows $c+d=0$ and then
 \begin{eqnarray}
 \label{2.11}
Tx \otimes f \phi (T)=\phi (T)x \otimes fT.
\end{eqnarray}
   This implies that $ \phi (T)x$ and $Tx$ are linearly dependent for all $x \in X$ such that $x \not \in \ker T$. So it is clear that $ \phi (T)x$ and $Tx$ are linearly dependent for all $x \in X$. Thus we can conclude from Theorem 2.3 in \cite{1} that there exists either a complex number $ \lambda _T$ such that $ \phi (T)= \lambda _TT $ for all non- rank one operator $T$ or $h,g \in X ^{ \prime}$ such that $T=x \otimes h$ and $ \phi (T)=x \otimes g$.
   By placing in equation \eqref{2.11}, we have
   $$x \otimes g=g(x)f(x)x \otimes h$$
   which implies $g= g(x)f(x)h$ and then for all operator $T$, there exists a complex number $ \lambda _T$ such that $ \phi (T)= \lambda _TT $.
   \end{proof}
\par \vspace{.3cm}
\textbf{Proof of Theorem 1.1.} By Lemma \ref{2.6}, it is enough to prove that $\lambda _T=1$. If $T$ is equal to $0$, $I$ or an idempotent, then by Lemmas \ref{2.1} and \ref{2.5}, we know that $\lambda _T=1$.

For any $T,S \in \mathcal{A}$ we have
$$\begin{pmatrix}
     TS & T \\
     -STS & -ST \\
   \end{pmatrix}^2=0$$
and then there exists $a,b$ such that
$$\begin{pmatrix}
   \phi(TS)-aI & \phi(T) \\
     \phi(-STS) & \phi(-ST)-aI \\
   \end{pmatrix}
   \begin{pmatrix}
   \phi(TS)-bI & \phi(T) \\
     \phi(-STS) & \phi(-ST)-bI \\
   \end{pmatrix}=0$$
   which by Lemma \ref{2.6} and Equation $(2.10)$
  $$ \begin{pmatrix}
     \lambda_{TS}TS-aI & \lambda_{T}T \\
     -\lambda_{STS}STS & -\lambda_{ST}ST-aI \\
   \end{pmatrix}
   \begin{pmatrix}
   \lambda_{TS}TS-bI & \lambda_{T}T \\
     -\lambda_{STS}STS & -\lambda_{ST}ST-bI \\
   \end{pmatrix}=0.$$
   This implies
  \begin{eqnarray}
 \label{2.12}
(\lambda_{TS}TS-aI)(\lambda_{TS}TS-bI)- \lambda_{T}\lambda_{STS}TSTS=0
\end{eqnarray}
and
   \begin{eqnarray}
 \label{2.13}
\lambda_{T}(\lambda_{TS}TS-aI)T-\lambda_{T}T(\lambda_{ST}ST+bI)=0.
\end{eqnarray}
Suppose $x \in \mathcal{X}$ and $f \in \mathcal{X}^ {\prime}$ are arbitrary elements. Hence there exists $y \in \mathcal{X}$ such that $y$ is linearly independent of $x$ and $f(y)=1$ and also there exists $g \in \mathcal{X}^ {\prime}$ such that $g(x)=1$ and $g(y)=1$. Setting $T=x \otimes f$ and $g=y \otimes g$ in Equation $(2.12)$, we get
   \begin{eqnarray*}ý
 \begin{array}{lcl}
   (\lambda_{x \otimes g}x \otimes g-aI)(\lambda_{x \otimes g}x \otimes g-bI)- \lambda_{x \otimes f}\lambda_{y \otimes g}x \otimes g=0\\
    \Rightarrow (1-a-b- \lambda_{x \otimes f}) x \otimes g+ abI=0
 ý\end{array}
 ý\end{eqnarray*}
which implies
\begin{eqnarray}
 \label{2.14}
ab=0,~~~ \lambda_{x \otimes f}=1-a-b.
\end{eqnarray}
Now with the same placement, that is $T=x \otimes f,~~g=y \otimes g$, in Equation $(2.13)$, we get
\begin{eqnarray*}ý
 \begin{array}{lcl}
   \lambda_{x \otimes f}(\lambda_{x \otimes g}x \otimes g-aI)x \otimes f-\lambda_{x \otimes f}x \otimes f(\lambda_{y \otimes f}y \otimes f+bI)=0\\
    \Rightarrow [\lambda_{x \otimes f}(1-a)- \lambda_{x \otimes f}(1+b)] x \otimes f=0
 ý\end{array}
 ý\end{eqnarray*}
which due to being non-zero of $ \lambda_{x \otimes f}$, we have
\begin{eqnarray}
 \label{2.15}
a=-b.
\end{eqnarray}
This together with $ab=0$ from $(2.14)$ implies $a=b=0$ and then again by using $(2.14)$, we obtain $\lambda _T=1$ for $T=x \otimes f$.

Suppose $0 \neq T \in \mathcal{A}$ and $x \in \mathcal{X}$ are arbitrary elements. Thus there exists a functional $f$ such that $f(Tx)=1$. Setting $S=x \otimes f$ in Equation $(2.12)$, we get
\begin{eqnarray*}ý
 \begin{array}{lcl}
   (Tx \otimes f-aI)(Tx \otimes f-bI)- \lambda_{T}Tx \otimes f=0\\
    \Rightarrow (1-a-b- \lambda_{T}) Tx \otimes f+ abI=0
 ý\end{array}
 ý\end{eqnarray*}
which implies
\begin{eqnarray}
 \label{2.16}
ab=0,~~~ \lambda_{x \otimes f}=1-a-b.
\end{eqnarray}
Setting $S=x \otimes f$ in Equation $(2.13)$, we get
\begin{eqnarray*}ý
 \begin{array}{lcl}
   \lambda_{T}(Tx \otimes f-aI)T-\lambda_{T}T(x \otimes fT+bI)=0\\
    \Rightarrow -\lambda_{T}(a+b)T=0
 ý\end{array}
 ý\end{eqnarray*}
which due to being non-zero of $ \lambda_{T}$, we have
\begin{eqnarray}
 \label{2.15}
a=-b.
\end{eqnarray}
This together with $ab=0$ from $(2.16)$ implies $a=b=0$ and then again by using $(2.16)$, we obtain $\lambda _T=1$ for $T\neq 0$.
The proof is complete.
\par \vspace{.3cm}
\textbf{Proof of Theorem 1.2.} Using the result in \cite{9} concerning characterizing
maps on idempotents matrices, the assertion can be proved by a similar
argument as in the proof of Theorem 1.2.

\bibliographystyle{amsplain}

\end{document}